\documentclass[12pt]{amsart}
\usepackage[english]{babel}
\usepackage{nicefrac}
\usepackage{hyperref}
\usepackage{orcidlink}
\usepackage{url}
\usepackage{mathtools}
\mathtoolsset{showonlyrefs}
\usepackage{anysize}
\usepackage{geometry}
\marginsize{2.0cm}{2.0cm}{2.0cm}{2.0cm}
\usepackage{graphicx}
\usepackage{tikz}
\usepackage{atbegshi}

\theoremstyle{plain}
\numberwithin{equation}{section}
\newtheorem{theorem}{Theorem}[section]
\newtheorem{lemma}[theorem]{Lemma}
\newtheorem{proposition}[theorem]{Proposition}
\theoremstyle{remark}
\newtheorem{remark}[theorem]{Remark}
\newcommand{\ve}{\varepsilon}
\newcommand{\ud}{\mathrm{d}}

\newcommand{\cN}{\mathcal{N}}
\newcommand{\tv}{\mathrm{d}_{\mathrm{TV}}}
\newcommand{\Wp}{\mathcal{W}_p}
\newcommand{\ii}{\mathsf{i}}
\newcommand{\tf}{\mathfrak{t}}

\begin{document}

\AtBeginShipoutFirst{%
    \begin{tikzpicture}[remember picture, overlay]
        \node[anchor=north west, xshift=1.5cm, yshift=-2.8cm] at (current page.north west) {%
            \includegraphics[width=4cm]{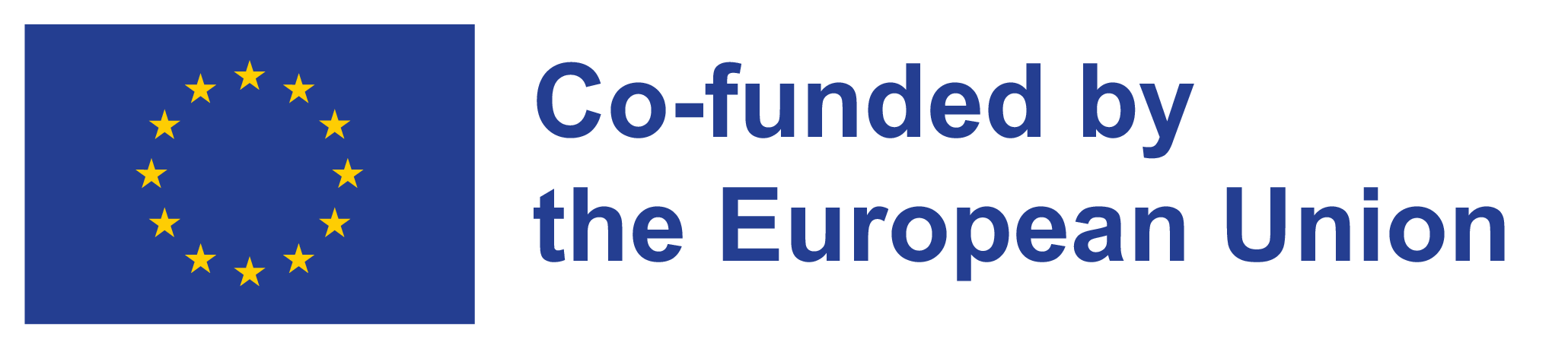} 
        };
    \end{tikzpicture}%
}

\title[Cut-off phenomenon for multivariate general linear processes]{Cut-off phenomenon and asymptotic mixing for multivariate general linear processes}
\author{Gerardo Barrera\orcidlink{0000-0002-8012-2600}}
\address{Center for Mathematical Analysis, Geometry and Dynamical Systems, Mathematics Department, Instituto Superior T\'ecnico, Universidade de Lisboa, 1049-001, Av. Rovisco Pais 1, 1049-001 Lisboa, Portugal.}
\email{gerardo.barrera.vargas@tecnico.ulisboa.pt}

\author{Michael A. H\"ogele\orcidlink{0000-0001-5744-0494}}
\address{Departamento de Matem\'aticas, Universidad de los Andes, Bogot\'a, Colombia.}
\email{ma.hoegele@uniandes.edu.co}

\author{Pauliina Ilmonen\orcidlink{0000-0002-8992-0553}}
\address{Department of Mathematics and Systems Analysis, Aalto University School of Science, Espoo, Finland.}
\email{pauliina.ilmonen@aalto.fi}

\author{Lauri Viitasaari\orcidlink{0000-0002-3015-8664}}
\address{Department of Information and Service Management, Aalto University School of Business, Espoo, Finland.}
\email{lauri.viitasaari@aalto.fi}

\keywords{Asymptotic mixing times; Cut-off phenomenon; Fractional Brownian motion; Gaussian processes; Limiting behavior; Ornstein--Uhlenbeck process; Total variation distance; Wasserstein distance}

\subjclass[2000]{Primary 60G10, 60G15, 37A25; Secondary  60G18, 60H05, 60G22}

\begin{abstract}
The small noise cut-off phenomenon in continuous time and space has been studied in the recent literature for the linear and non-linear stable Langevin dynamics with additive L\'evy drivers - understood as abrupt thermalization of the system along a particular time scale to its dynamical equilibrium - both for the total variation distance and the Wasserstein distance. The main result of this article establishes sufficient conditions for the window and profile cut-off phenomenon, which are flexible enough to cover the renormalized (non-Markovian) Ornstein--Uhlenbeck process driven by fractional Brownian motion and a large class of Gaussian and non-Gaussian, homogeneous and non-homogeneous drivers with (possible) finite second moments. The sufficient conditions are stated both for the total variation distance and the Wasserstein distance. Important examples are the multidimensional fractional Ornstein--Uhlenbeck process, the empirical sampling process of a fractional Ornstein--Uhlenbeck process, an Ornstein--Uhlenbeck processes driven by an Ornstein--Uhlenbeck process and the inhomogeneous Ornstein--Uhlenbeck process arising in simulated annealing.
\end{abstract}

\maketitle

\tableofcontents

\section{Introduction}\label{sec:intro}

The cut-off phenomenon - also known as abrupt convergence - describes an ergodic threshold phenomenon for a parameter dependent stochastic processes, 
and goes back to the studies by D.~Aldous, P.~Diaconis and M.~ Shahshahani  at the beginning of the 1980ies on strong stopping times of random walks on groups and card shuffling Markov models, see~\cite{Al83, Al89, AD87, AD, BD92, DGM90, DS, DI}.
Essentially, it describes the asymptotically ever sharper decay of the distance between the current state of a stochastic process and its limiting dynamical equilibrium along a specific time scale. 
Recently, the cut-off phenomenon has been studied in different applied context including: stochastic simulations of chemical kinetics in~\cite{Bayati--Owahi--Koumoutsakos},
quantum Markov chains in~\cite{Kastoryano12},
coagulation-fragmentation systems (Becker--D\"oring equations) in~\cite{Pego17},
open quadratic fermionic systems in~\cite{Vernier},
deep limits for neural networks in~\cite{Avelin--Karlsson}, 
random quantum circuits/random quantum states in~\cite{Oh--Kais}, and
viscous energy shell model in~\cite{BHPPJSP}.

In the sequel, let $\ve>0$ denote a complexity parameter.
For a given initial datum $x$ we consider a continuous time stochastic process
$X^\ve:=(X^\ve_t(x))_{t\geq 0}$
which possesses a limiting measure $\mu^\ve$ as $t\to \infty$, and a renormalized family of distances or discrepancies $d = (d_\ve)_{\ve>0}$ on the set of probability measures over the respective state space of $X^\ve$. 
Then the family of processes $(X^\ve)_{\ve>0}$ is said to exhibit a profile cut-off phenomenon if there exist a cut-off time scale $t_\ve:=t_{\ve}(x)$ and time window $w_{\ve}:=w_{\ve}(x)>0$ satisfying $t_{\ve}\to \infty$ as $\ve\to 0^+$, $w_{\ve}=\textrm{o}(t_{\ve})$ as $\ve\to 0^+$, and
\begin{equation}\label{e:cutoff}
\lim_{\ve \to 0^+} d_{\ve}\left(X^{\ve}_{t_{\ve} +r\cdot w_{\ve}}(x), \mu^{\ve}\right) 
= \mathcal{P}_x(r)\quad \textrm{ for any }\quad r\in \mathbb{R},
\end{equation}
where $\mathcal{P}_x:\mathbb{R}\to (0,\infty)$ is a continuous 
function - the so-called 
cut-off profile - such that 
\begin{equation}
0 = \lim_{\rho \to \infty} \mathcal{P}_x(\rho )\leq \mathcal{P}_x(r) \leq \lim_{\rho \to -\infty}  \mathcal{P}_x(\rho ) = \mbox{diam}(d)\quad \textrm{ for all }\quad r\in \mathbb{R},
\end{equation}
where $\mbox{diam}(d):= \limsup_{\ve \to 0^+} \mbox{diam}(d_\ve)\in (0,\infty]$, with $\mbox{diam}(d_\ve)$ denoting the diameter with respect to $d_\ve$.

The cut-off phenomenon understood as~\eqref{e:cutoff} or possibly weaker versions of it (for instance, where the limit in~\eqref{e:cutoff} is replaced by appropriate upper and lower limits, in case it does not exist) is a very active field of research in the mathematics literature 
including:
birth and death chains in~\cite{BY1}, sampling processes in~\cite{BLY06,BAR--YCA,Diatta--Ngom,Diedhiou--Ngom}, 
reversible Markov chains in~\cite{BASU}, families of ergodic Markov processes on  manifolds in~\cite{Arnaudon--Coulibaly--Miclo,CSC08,Meliot14}, numerical analysis of sampling Markov chains in~\cite{Jonsson97},
asymmetric simple exclusion processes in~\cite{LL19},
simple exclusion process on the circle in~\cite{La16},
Ising model on the lattice in~\cite{LubetzkySly13},
non-negatively curved Markov chains in~\cite{Salez}, 
Markov processes under information divergences in~\cite{Wang--Choi},
among others.
For an introduction to the field we refer to~\cite{DI,LPW}. In discrete state space, the distance $d_\ve$ is often given as the non-normalized total variation distance. 
In continuous state spaces and continuous time, the cut-off phenomenon was first studied for the small noise limit of 
linear and non-linear multivariate overdamped Langevin equation driven by small Brownian motion~\cite{BAR2018,BJ, BJ1}. 
For more general L\'evy drivers (including degenerate Brownian motion) the cut-off phenomenon has been established for finite and infinite dimensional stable systems, see~\cite{BAR2021,BHPTV,BARESQ, BPJC,BALIU,Insuk}. While the total variation distance requires a certain amount of smoothing properties of the underlying noise in combination with the linearized structure close to the stable state (such as the Kalman rank condition), 
the use of the renormalized Wasserstein distance, which does not depend on this regularity, allows to include degenerate physical systems such as Jacobi chains of oscillators with isolated heat bath
and even infinite dimensional systems~\cite{BH23,BHPGBM,BHPWA,BHPWANO,BHPSPDE, BHPPJSP}.

All these systems have in common that due to the essentially Markovian structure combined with some compactness/tightness properties of the process, they exhibit a convergence to a dynamical equilibrium. To our knowledge the cut-off phenomenon of a fractional Brownian motion driver and hence continuous-time non-Markovian systems has not been established to date.
By means of specific examples and using coupling techniques it is argued in~\cite{Konstantopoulos--Baccelli} that the cut-off phenomenon is also valid even for systems without any Markovian structure. 
The novelty of this article is to establish the cut-off phenomenon for a large class of families of non-Markovian processes with a stable linear component. 

One of the simplest such systems is a small noise Ornstein--Uhlenbeck process driven by fractional Brownian motion, introduced and studied in~\cite{Cheridito}. Fractional Brownian motion, introduced by B. Mandelbrot and J.~W. van Ness in~\cite{MvN} for a given Hurst-parameter $H\neq \nicefrac{1}{2}$ shares many properties with the Brownian motion such as Gaussianity and stationarity of the increments. However, it is not a semi-martingale nor Markovian. For details on fractional Brownian motion and its applications, see the monograph~\cite{Mishura} and the references therein. On top of the case of the fractional Ornstein--Uhlenbeck process, we consider generalized Ornstein--Uhlenbeck processes driven by non-Markovian and non-Gaussian drivers. These cover essentially all stationary processes in continuous time, 
see~\cite{Viitasaari} for the one-dimensional case and~\cite{Voutilainen et alt} for the extension to multivariate setting. For reader's convenience, in Section~\ref{s:1dfOU} we explain in detail 
how to establish cut-off in the sense of~\eqref{e:cutoff} for the one-dimensional fractional Ornstein--Uhlenbeck process in the total variation distance and in the renormalized Wasserstein distance. For $H = \nicefrac{1}{2}$ one then recovers  the original Ornstein--Uhlenbeck process, leading to the cut-off results 
in~\cite{BAR2018, BHPWA, BPJC} in various settings. 
More generally, our main results given in Section~\ref{sec:mainresults} state the following. 
For a general class of multivariate linear processes $(X^\ve_t(x))_{t\ge 0}$ such that for each $t>0$ the marginal at time $t$ satisfies
\begin{equation}
\label{eq:X-intro}
X^\ve_t(x)= e^{- \Lambda t}x + \ve S_t\quad  \textrm{in law},
\end{equation} 
where $\ve>0$, $-\Lambda\in \mathbb{R}^{m\times m}$ is Routh--Hurwitz stable and $x\in \mathbb{R}^m$, $x\neq 0_m$.  
In addition, for $d_0$ being the total variation distance or the Wasserstein distance of order $p\geq 1$
we assume that the following limit is valid:
\[
\lim_{t \to \infty} d_0\left(\frac{S_t}{\sigma_t}, Z\right) = 0
\]
for some sub-exponential (possibly constant) deterministic scale $(\sigma_t)_{t\geq 0}$ and a random variable $Z$ taking values on $\mathbb{R}^m$. We aim to show the cut-off phenomenon in the sense of the existence of some cut-off time scale $t_\ve \to \infty$ as $\ve\to 0$, time window $w_\ve=\mathrm{o}(t_\ve)$ as $\ve \to 0$, satisfying
\[
\lim_{\ve \to 0} d_\ve\left(\frac{X_{t_\ve+r\cdot w_\ve}(x)}{\sigma_{t_\ve+r\cdot w_\ve}}, Z\right) = \mathcal{P}_x(r) \quad \textrm{ for all }\quad r\in \mathbb{R}
\]
for 
$d_\ve$ being the total variation distance or the $(\nicefrac{1}{\ve})-$renormalized  Wasserstein distance of order $p\geq 1$.
In Section~\ref{sec:examples}, we showcase the applicability of our results by giving several new examples that have not yet been considered in the literature. In addition, if $S_t = \int_0^t e^{-\Lambda (t-s)} \ud D_s$  and $(D_t)_{t\geq 0}$ is a Brownian motion or a discontinuous L\'evy process, we recover some important known cases. Our new examples show that the empirical sample mean process of a sequence of fractional Ornstein--Uhlenbeck processes exhibits a cut-off phenomenon as the sample size increases. We also show that Ornstein--Uhlenbeck processes with general stationary increment drivers in $L^2$ and the Ornstein--Uhlenbeck process driven by another centered Gaussian Ornstein--Uhlenbeck process can be treated by applying our results. Moreover, we apply the results to the case of inhomogeneous stochastic convolution $S$ that arise naturally in simulated annealing context. Finally, we treat the case of the non-centered Ornstein--Uhlenbeck drivers by considering a particular nontrivial rescaling of the process $S$. These examples illustrate the wide generality of our results.

While it comes not as a surprise that different distances measure different details of discrepancy, we stress two distinctive features of the total variation distance and the Wasserstein distance: 
The total variation distance enjoys the highly non-linear, zero-homogeneity structure that $\tv(c X, c Y) = \tv(X, Y)$ for any $c\neq 0$ (Lemma~\ref{lem:ptv}(ii)). 
Even the linear shifts of a random vector by a sequence of random vectors with discrete laws which converge in distribution are in general discontinuous under the total variation
distance~\cite[Lemma~1.17]{BHPTV}. Therefore, its topology is finer than the one by convergence in distribution. 
In our main results it therefore captures certain features of the driving noise in the cut-off profile such as the variance, see 
formula~\eqref{e:totalvariationprofile}. As expected, due to the zero-homogeneity no renormalization of the process $S$ is needed. On the other hand, the Wasserstein distance, is more aligned with the linear structure via the so-called shift-additivity property, $\Wp(v+X,X)=\|v\|$, $v\in \mathbb{R}^m$ deterministic vector and $p\geq 1$,
see Lemma~\ref{lem:basicwp}(iii). Here the particular features of the noise essentially cancel out and the profile is given by the norm of the matrix exponential of our linear process, which implies that the cut-off profile is universal in the sense that it does not retain any properties of the driving noise but only of the deterministic linear part. Moreover, the distance has to be renormalized by the noise intensity $\ve$ in order to exhibit a cut-off phenomenon. 

The article is organized as follows. 
In Section~\ref{s:1dfOU}, we explain the cut-off phenomenon for the scalar fractional Ornstein--Uhlenbeck process in total variation and in the Wasserstein distance.
In Section~\ref{sec:mainresults} this is generalized to sufficient ergodic convergence conditions in total variation and in the Wasserstein distance on the underlying multivariate processes of type \eqref{eq:X-intro}. In particular, this covers the case of multivariate convolution processes for which the stochastic convolution can be defined under mild conditions. For those systems we formulate sufficient conditions for window cut-off and characterize profile cut-off, both in the total variation and in the  Wasserstein distance.
In Section~\ref{sec:examples} we apply these results to the examples laid out before. The appendix is divided into three parts. 
Appendix~\ref{ap:proofs} contains the proofs of the main results. In Appendix~\ref{ap:tv} and Appendix~\ref{ap:wp} we gather the most relevant properties of the total variation and the Wasserstein distance. 

\section{Univariate fractional 
Ornstein--Uhlenbeck process}\label{s:1dfOU}

In this section we explain our main results for reader's convenience in the  paradigmatic example of the univariate fractional Ornstein--Uhlenbeck process.
For a given $\ve>0$, let $(X^{\ve,H}_t(x))_{t\geq 0}$ be the solution of the
Ornstein--Uhlenbeck process driven by fractional Brownian motion (for short OUfBM)
 given by
\begin{equation}
\label{modelo}
\ud X^{\ve,H}_t = -\lambda X^{\ve,H}_t \ud t+\ve \ud B^{H}_t
\quad\textrm{ for } \quad t\geq  0,\quad\textrm{ and } \quad
X^{\ve,H}_0 = x,
\end{equation} 
where $x\in \mathbb{R}$, $x\neq 0$, is a deterministic initial datum, 
$\lambda>0$,  and the driven noise 
$(B^{H}_t)_{t\geq 0}$ is a fractional Brownian motion on $\mathbb{R}$ with Hurst exponent
$H\in (0,1)$.
We recall that $(B^H_t)_{t\geq 0}$ is a continuous-time zero-mean Gaussian process starting at zero almost surely, having stationary increments, and whose covariance function is given by
\begin{equation}
\mathbb{E}\left[B^H_sB^H_t\right]=\frac{1}{2}\left(t^{2H}+s^{2H}-|t-s|^{2H}\right)\quad \textrm{ for all }\quad s,\,t\geq 0.
\end{equation}
For $H=1/2$  the classical Brownian motion is recovered and then its increments are independent. However, for $H\neq 1/2$, the increments of $(B^H_t)_{t\geq 0}$ are not independent.
By the variation of constants formula we have the path-wise solution of~\eqref{modelo} 
\[
X^{\ve,H}_t(x)=e^{-\lambda t}x+\ve e^{-\lambda t}\int_{0}^{t}e^{\lambda s} \ud B^{H}_s \quad \textrm{ for all }\quad  t\geq 0.
\]
We define the stationary process $U^H:=(U^H_t)_{t\geq 0}$  by
\[
U^H_t:= e^{-\lambda t}\int_{-\infty}^t e^{-\lambda(t-s)}\ud B^{H}_s \quad \textrm{ for }\quad  t\geq 0.
\]
Note that $U^H_t\stackrel{\mathrm{Law}}{=}U^H_0$, where 
$\stackrel{\mathrm{Law}}{=}$ denotes equality in distribution, and that  $U^H_0$ has mean zero.
Then we denote the covariance of $U^H$ by $\mathcal{R}_{U^H}(t)$, that is,
$\mathcal{R}_{U^H}(t)=\mathbb{E}[U^H_0U^H_t]$, $t\geq 0$.
We then have the following representation of the solution
\begin{equation}
\begin{split}
X^{\ve,H}_t(x)\stackrel{\mathrm{Law}}{=}& e^{-\lambda t}(x-\ve U^H_0)+\ve U^H_t= 
e^{-\lambda t}x +\ve\left( U^H_t-e^{-\lambda t}U^H_0\right)\quad
\textrm{ for }\quad t\geq 0.
\end{split}
\end{equation}
For $m\in \mathbb{R}$ and $\sigma^2>0$ we denote the Normal distribution with mean $m$ and variance $\sigma^2$ by
$\cN(m,\sigma^2)$.

\begin{lemma}[Ergodicity]\label{lem:ergasymp}
\hfill

\noindent
Let $H\in (0,1)$ be fixed.
For all $t>0$, the random variable $X^{\ve,H}_t(x)$ has Normal distribution with mean 
\begin{equation}
m_t(x):=e^{-\lambda t}x
\end{equation}
and variance
\begin{equation}
\sigma^2(t;\ve,H)=\ve^2\left(\mathcal{R}_{U^H}(0) + e^{-2\lambda t}\mathcal{R}_{U^H}(0)-2e^{-\lambda t}\mathcal{R}_{U^H}(t)\right),
\end{equation}
where 
\begin{equation}\label{e:Gammavariance}
\mathcal{R}_{U^H}(0)=\mathbb{E}[U^2_0]=\frac{\lambda}{2}\int_0^\infty e^{-  \lambda s}s^{2H}\ud s = \frac{\lambda^{-2H}\Gamma(2H+1)}{2},
\end{equation}
and $\Gamma$ denotes the usual Gamma function.
Moreover, as $t\to \infty$ the law of the random variable $X^{\ve,H}_t(x)$ converges (in the total variation distance and in the Wasserstein distance of any order $p\geq 1$) to a random variable 
$X^{\ve,H}_\infty$ that has Normal distribution with zero mean and variance $\mathcal{R}_{U^H}(0)$.

In particular, for $H=1/2$ it follows that
$\mathcal{R}_{U^H}(0)=1/(2\lambda)$ and 
$\mathcal{R}_{U^H}(t)=\mathcal{R}_{U^H}(0)e^{-\lambda t}$, $t\geq 0$, while for $H\neq 1/2$ the autocorrelation function
$\mathcal{R}_{U^H}(t)$ is given as special function, 
see~Theorem~2.3~in~\cite{Cheridito}, which has behavior
$\mathcal{R}_{U^H}(t)\sim t^{2H-2}$ as $t$ tends to infinity.
\end{lemma}

\noindent In what follows, we study the sharp behavior of the discrepancy
\begin{equation}
\ud^{\ve,H}_{\mathrm{TV}}(t;x):=\tv\left(X^{\ve,H}_t(x),
X^{\ve,H}_\infty\right),\quad t\geq 0, x\in \mathbb{R}, 
\end{equation}
where $\tv(\cdot,\cdot)$ denotes the total variation distance. We refer to Appendix~\ref{ap:tv} for the definition and basic properties of the total variation distance.
The main result of this section is the following.
\begin{theorem}[Cut-off convergence in the total variation distance]\label{th:tv}
\hfill

\noindent
Let $H\in (0,1)$ be fixed and $x\in \mathbb{R}$, $x\neq 0$, be fixed.
For any $\ve\in (0,1)$ 
define 
\begin{equation}\label{e:cutofftimescale}
t_{\ve}:=\frac{1}{\lambda}\ln\left(\frac{1}{\ve}\right).
\end{equation}
Then for any $(w_{\ve})_{\ve\in(0,1)}$ such that $w_{\ve}\to w$ as $\ve\to 0^+$ with $w>0$ it follows that
\begin{equation}\label{eq:resulttv}
\lim\limits_{\ve\to 0^+}
\ud^{\ve,H}_{\mathrm{TV}}(t_{\ve}+r\cdot w_{\ve};x)=
\mathcal{P}^{H}_{\mathrm{TV}}(r),\quad r\in \mathbb{R},
\end{equation}
where
\begin{equation}
\mathcal{P}^{H}_{\mathrm{TV}}(r):=\tv\left(
\cN\left(e^{-\lambda r w}x,\mathcal{R}_{U^H}(0)
\right),
\,
\cN\left(0,\mathcal{R}_{U^H}(0)\right)\right),\quad r\in \mathbb{R}.
\end{equation}
In addition, for $x\in \mathbb{R}\setminus\{0\}$ the following limits holds true:
\begin{equation}\label{eq:TV10}
\lim\limits_{r\to -\infty} \mathcal{P}^{H}_{\mathrm{TV}}(r)=1
\quad
\textrm{ and } 
\quad
\lim\limits_{r\to \infty} \mathcal{P}^{H}_{\mathrm{TV}}(r)=0.
\end{equation}
\end{theorem}

\begin{remark}[Total variation profile in terms of the error function]\hfill

\noindent
Using Item~(i)~in~Lemma~B.2~in~Appendix~B of~\cite{BALIU} 
and Item~(ii)~in~Lemma~\ref{lem:ptv}
we obtain that, for every $r\in \mathbb{R}$,
\begin{equation}\label{e:totalvariationprofile}
\begin{split}
\mathcal{P}^{H}_{\mathrm{TV}}(r)=
\frac{2}{\sqrt{2\pi}}\int_{0}^{\frac{1}{2\sqrt{\mathcal{R}_U(0)}}e^{-\lambda r w}|x|} e^{-\frac{y^2}{2}} \ud y,
\end{split}
\end{equation}
which with the help of Dominated Convergence Theorem  yields~\eqref{eq:TV10}.
\end{remark}

\begin{proof}[Proof of Theorem~\ref{th:tv} (Cut-off convergence in the total variation distance)]
\hfill

\noindent
By Lemma~\ref{lem:ergasymp} and Item~(ii)~in~Lemma~\ref{lem:ptv}
we have
\begin{equation}
\begin{split}
\ud^{\ve,H}_{\mathrm{TV}}(t;x)&=\tv\left(X^{\ve,H}_t(x),X^{\ve,H}_\infty\right)\\
&=
\tv\left(
\cN\left(e^{-\lambda t}x,
\ve^2\left(\mathcal{R}_{U^H}(0) + e^{-2\lambda t}\mathcal{R}_{U^H}(0)-2e^{-\lambda t}\mathcal{R}_{U^H}(t)\right)
\right),
\,\cN\left(0,\ve^2 \mathcal{R}_{U^H}(0)\right)\right)\\
&=
\tv\left(
\cN\left(\frac{e^{-\lambda t}x}{\ve},
\mathcal{R}_{U^H}(0) + e^{-2\lambda t}\mathcal{R}_{U^H}(0)-2e^{-\lambda t}\mathcal{R}_{U^H}(t)
\right),
\,\cN\left(0, \mathcal{R}_{U^H}(0)\right)\right).
\end{split}
\end{equation}
The preceding equality with the help of Lemma~B.4 in Appendix~B of~\cite{BALIU} implies~\eqref{eq:resulttv}.
\end{proof}

\noindent In the sequel, we study the  behavior of
\begin{equation}
\ud^{\ve,H}_{\Wp}(t;x):=\frac{1}{\ve}
\Wp\left(X^{\ve,H}_t(x),
X^{\ve,H}_\infty\right),\quad t\geq 0
\end{equation}
for $p\geq 1$,
where $\Wp(\cdot,\cdot)$ denotes the Wasserstein distance of order $p$. We refer to Appendix~\ref{ap:wp} for the definition and basic properties of the Wasserstein distance.
The second main result of this section is the following.
\begin{theorem}[Cut-off convergence in the Wasserstein distance of order $p\geq 1$]\label{th:Wp}
\hfill

\noindent
Let $H\in (0,1)$ be fixed and $x\in \mathbb{R}$ be fixed.
For any $\ve\in (0,1)$ 
define $t_{\ve}$ as in~\eqref{e:cutofftimescale}.
Assume that $p\geq 1$.
For any $(w_{\ve})_{\ve\in(0,1)}$ such that $w_{\ve}\to w$ as $\ve\to 0^+$ with $w>0$ it follows that
\begin{equation}\label{eq:resultWp}
\lim\limits_{\ve\to 0^+}
\ud^{\ve,H}_{\Wp}(t_{\ve}+r\cdot w_{\ve};x)=
\mathcal{P}^{H}_{\Wp}(r),\quad r\in \mathbb{R},
\end{equation}
where
\begin{equation}
\mathcal{P}^{H}_{\Wp}(r):=\Wp\left(
\cN\left(e^{-\lambda r w}x,
\mathcal{R}_{U^H}(0)
\right),
\,\cN\left(0,\mathcal{R}_{U^H}(0)\right)\right),\quad r\in \mathbb{R}.
\end{equation}
In addition, for $x\in \mathbb{R}\setminus\{0\}$ the following limits holds true
\begin{equation}\label{eq:WA10}
\lim\limits_{r\to -\infty} \mathcal{P}^{H}_{\Wp}(r)=\infty
\quad
\textrm{ and } 
\quad
\lim\limits_{r\to \infty} \mathcal{P}^{H}_{\Wp}(r)=0.
\end{equation}
\end{theorem}

\begin{remark}[Wasserstein profile for $p\geq 1$ in terms of the exponential function]\hfill

\noindent
By Item(iii)~in~Lemma~\ref{lem:basicwp} for $p\geq 1$ and every $r\in \mathbb{R}$ it follows that
$\mathcal{P}^{H}_{\Wp}(r)=e^{-\lambda r w}|x|$,
from which we obtain~\eqref{eq:WA10}. 
We observe that $\mathcal{P}^{H}_{\Wp}(r)$ does not depend on $p$
and is in this sense universal among all Wasserstein distances of order $p\geq 1$.
\end{remark}

\begin{proof}[Proof of 
Theorem~\ref{th:Wp} (Cut-off convergence in the Wasserstein distance of order $p\geq 1$)]
\hfill

\noindent
By Lemma~\ref{lem:ergasymp} and Item~(ii)~in Lemma~\ref{lem:basicwp} 
we have
\begin{equation}
\begin{split}
\ud^{\ve,H}_{\Wp}(t;x)&=\frac{1}{\ve}\Wp\left(X^{\ve,H}_t(x),X^{\ve,H}_\infty\right)\\
&=
\frac{1}{\ve}
\Wp\left(
\cN\left(e^{-\lambda t}x,
\ve^2\left(\mathcal{R}_{U^H}(0) + e^{-2\lambda t}\mathcal{R}_{U^H}(0)-2e^{-\lambda t}\mathcal{R}_{U^H}(t)\right)
\right),
\,\cN\left(0,\ve^2 \mathcal{R}_{U^H}(0)\right)\right)\\
&=
\Wp\left(
\cN\left(\frac{e^{-\lambda t}x}{\ve},
\mathcal{R}_{U^H}(0) + e^{-2\lambda t}\mathcal{R}_{U^H}(0)-2e^{-\lambda t}\mathcal{R}_{U^H}(t)
\right),
\,\cN\left(0, \mathcal{R}_{U^H}(0)\right)\right).
\end{split}
\end{equation}
The preceding equality with the help of Item~(iv) in~Lemma~\ref{lem:basicwp} implies~\eqref{eq:resultWp}.
\end{proof}

\section{Cut-off convergence for general linear processes in law}\label{sec:mainresults}

In this section, we establish the small noise cut-off phenomenon in the total variation distance and 
in the Wassertein distance of order $p\geq 1$
for general linear process in law of the type~\eqref{d:X} given below.
In the sequel, $m$ denotes a fixed positive integer,
$\langle \cdot , \cdot \rangle $ denotes the standard inner product for $\mathbb{R}^m$,
and
 $\|\cdot\|$ denotes the Euclidean norm in $\mathbb{R}^m$.
Moreover, we always assume that $-\Lambda\in \mathbb{R}^{m\times m}$ is
a Routh--Hurwitz matrix. That is, all the eigenvalues of $-\Lambda$ have negative real parts.
For each $\ve>0$ and $x\in \mathbb{R}^m$
we consider the following linear continuous-time stochastic process $X^{\ve,x}:=(X^{\ve}_t(x))_{t\geq 0}$ on $\mathbb{R}^m$ in law defined by
\begin{equation}\label{d:X}
X^{\ve}_t(x)\stackrel{\mathrm{Law}}{=} e^{-\Lambda t}x+\ve S_t\quad \textrm{ for all }\quad t\geq 0,
\end{equation}
where $S=(S_t)_{t\geq 0}$ is an arbitrary continuous-time stochastic process taking values on $\mathbb{R}^m$ and satisfying $\mathbb{E}[S_t]=0_m$ for all $t\geq 0$, with $0_m$ denoting the  zero vector in $\mathbb{R}^m$.
For classes of examples of processes $S$ we refer to Section~\ref{sec:examples}.
We observe that
$\mathbb{E}[X^{\ve}_t(x)]=e^{-\Lambda t}x$ for all $x\in \mathbb{R}^m$ and $t\geq 0$.
For each $x\in \mathbb{R}^m\setminus\{0_m\}$ we study the
asymptotic behavior of $e^{-\Lambda t}x$ for $t\gg 1$.

\begin{lemma}[Hartman–Grobman asymptotics]\label{lem:Har}
\hfill

\noindent
Let $-\Lambda\in \mathbb{R}^{m\times m}$ be a Routh–Hurwitz matrix. For $x\in \mathbb{R}^m\setminus\{0_m\}$ there exist
\begin{itemize}
\item[(1)] a positive rate $\lambda:=\lambda(x)$,
\item[(2)]  size and multiplicity parameters 
$\ell:=\ell(x)$ and $m_*:=m_*(x)$, respectively, 
satisfying
$\ell,m_*\in \{1,\ldots,m\}$,
\item[(3)]
angular velocities $\theta_1:=\theta_{1}(x),\ldots,\theta_{m_*}:=\theta_{m_*}$ in $\mathbb{R}$ such that whenever $\theta_j\neq 0$ we have $\theta_{j+1}=-\theta_{j}$,
\item[(4)] 
linearly independent vectors 
$v_1(x):=v_1,\ldots, v_{m_*}(x):=v_{m_*}$ in $\mathbb{C}^{m}$,
\end{itemize}
such that
\begin{equation}
\lim\limits_{t\to \infty}
\left\|
\frac{e^{\lambda t}}{t^{\ell-1}}e^{-\Lambda t}x-v(t;x)  \right\|=0,
\end{equation} 
where 
\begin{equation}\label{ec:vdef}
v(t;x):=\sum_{j=1}^{m_*} e^{\ii \theta_j t}v_{j}\quad \textrm{ for all }\quad  t\geq 0.
\end{equation}
Moreover, the following  holds true:
\begin{equation}\label{eq:0bound}
0<\liminf\limits_{t\to \infty}  \|v(t;x)\|\leq \limsup\limits_{t\to \infty}  \|v(t;x)\|<\infty.   
\end{equation}
\end{lemma}
A version of this lemma is established as Lemma~B.1 in Appendix~B of~\cite{BJ1}. In~\cite{BJ1}, the result was proven under an additional coercivity assumption for $\Lambda$. That is, in~\cite{BJ1} it is assumed that there exists $\delta>0$ satisfying 
$\langle z,\Lambda z \rangle\geq \delta \|z\|^2$ for all $z\in \mathbb{R}^m$.
However, inspecting the proof line by line one can easily see that the authors only use the fact that the matrix $-\Lambda$ is a Routh--Hurwitz matrix.

Lemma~\ref{lem:Har} states roughly that the rate of the long-term dynamics is determined by $\lambda$, the total number and the size (in the sense of geometric multiplicity) of those Jordan blocks of $\Lambda$ with eigenvalues which have precisely the (smallest) real part $\lambda$, while the bounded limiting object is determined exclusively by the angular velocities, that is the imaginary parts of the aforementioned eigenvalues.
Also, we emphasize that the angular velocities $\theta_1,\ldots,\theta_{m_*}$ take values in $\mathbb{R}$ instead of taking values on $[0,2\pi)$, due to the fact that the map $t\mapsto e^{\ii \theta t}$ is not $2\pi$-periodic for $\theta\not\in \mathbb{Z}$. 
Compared to the results given in~\cite{BJ1}, the present setting of Lemma~\ref{lem:Har} allows to include more examples of interest such as Jacobi chains of oscillators, the Brownian gyrator, or more general linear stable networks, see for instance~\cite{BH23,BHPWA}. 

\noindent For $x\in \mathbb{R}^m\setminus\{0_m\}$
we define the $\omega$-limit set of $x$ under the continuous-time deterministic dynamics  $(v(t;x))_{t\geq 0}$ given in~\eqref{ec:vdef} by
\begin{equation}\label{eq:basin}
\begin{split}
\omega(x):&=\left\{v\in \mathbb{R}^m: \textrm{$v$ is an accumulation point of } (v(t;x))_{t\geq 0}\right\}, 
\end{split}
\end{equation}
that is, $v\in \omega(x)$ if and only if 
there exists a sequence $(t_n)_{n\in \mathbb{N}}$ such that $t_n\to\infty$ and 
$v(t_n;x)\to v$ as $n\to \infty$.
Note that 
\begin{equation}\label{eq:novacio}
\omega(x)\neq \emptyset
\end{equation}
due to 
$\sup_{t\geq 0}\|v(t;x)\|\leq \sum_{j=1}^{m_*} \|v_{j}\|$
and the Bolzano--Weierstrass Theorem.
Moreover, 
\begin{equation}\label{eq:nozero}
0_m\not \in\omega(x)
\end{equation}
due to the lower bound in~\eqref{eq:0bound}.
We point out that for $m=1$, we have $m_*=1$, $\ell=1$, $\theta_1=\ldots \theta_{m_*}=0$, $\lambda=\lambda_1$ and $v_1=x$.

\noindent In the sequel, we define the cut-off time scale.
Let $x\in \mathbb{R}^d\setminus\{0_m\}$ be fixed.
For any $\ve\in (0,1)$ 
set
\begin{equation}\label{e:cutti}
\begin{split} 
t^*_\ve:=\frac{1}{\lambda}\ln\left(\frac{1}{\ve}\right)\quad \textrm{ and then define }\quad
t^{\mathsf{cut}}_{\ve}:=t^*_\ve+\left(\frac{\ell-1}{\lambda}\right)
\ln(\lambda t^*_\ve)-\frac{1}{\lambda}\ln\left(\sigma_{t^*_\ve}\right),
\end{split}
\end{equation}
where the $x$-depending constants $\lambda:=\lambda_x$ and  $\ell:=\ell_x$ 
are defined in Lemma~\ref{lem:Har}.

\subsection{Multivariate cut-off phenomenon in the total variation distance}

The first main result in this subsection  states sufficient conditions for the window cut-off in the total variation distance.

\begin{theorem}[Window cut-off convergence in the total variation distance]\label{th:tvgeneral}
\hfill

\noindent
Assume that there exists a deterministic function $\sigma:(0,\infty)\to (0,\infty)$ and a random vector $Z$ on $\mathbb{R}^m$
satisfying
\begin{equation}\label{eq:tvhyp}
\lim\limits_{t\to \infty} \tv\left(\frac{S_t}{\sigma_t},Z\right)=0.
\end{equation}
In addition, assume that
for any $r>0$ it follows that
$\lim_{t\to \infty} \frac{\sigma_{t}}{\sigma_{t+r}}=1$.
Then for any $(w_{\ve})_{\ve\in(0,1)}$ such that $w_{\ve}\to w$ as $\ve\to 0^+$ with $w>0$ it follows that
\begin{equation}\label{e:infsup}
\begin{split}
\liminf_{\ve\to 0^+}
\tv\left(\frac{X^{\ve}_{t^{\mathsf{cut}}_{\ve}+r\cdot w_{\ve}}(x)}{\ve\sigma_{t^{\mathsf{cut}}_{\ve}+r\cdot w_{\ve}}}, Z\right)&=\tv\left(\lambda^{1-\ell}e^{-\lambda r w}\check{v}(x)+Z,Z\right),\\
\limsup_{\ve\to 0^+}
\tv\left(\frac{X^{\ve}_{t^{\mathsf{cut}}_{\ve}+r\cdot w_{\ve}}(x)}{\ve\sigma_{t^{\mathsf{cut}}_{\ve}+r\cdot w_{\ve}}}, Z\right)&=\tv\left(\lambda^{1-\ell}e^{-\lambda r w}\hat{v}(x)+Z,Z\right)
\end{split}
\end{equation}
for some $\check{v}(x),\hat{v}(x)\in \omega(x)$, where $\omega(x)$ is given in~\eqref{eq:basin} and $t^{\mathsf{cut}}_{\ve}$
is defined in~\eqref{e:cutti}.
In addition, the following limits holds true:
\begin{equation}\label{eq:perfiltv1}
\begin{split}
\lim_{r\to -\infty}\liminf\limits_{\ve\to 0^+}
\tv\left(\frac{X^{\ve}_{t^{\mathsf{cut}}_{\ve}+r\cdot w_{\ve}}(x)}{\ve\sigma_{t^{\mathsf{cut}}_{\ve}+r\cdot w_{\ve}}}, Z\right)
&=1,\\
\lim_{r\to \infty}\limsup\limits_{\ve\to 0^+}
\tv\left(\frac{X^{\ve}_{t^{\mathsf{cut}}_{\ve}+r\cdot w_{\ve}}(x)}{\ve\sigma_{t^{\mathsf{cut}}_{\ve}+r\cdot w_{\ve}}}, Z\right)
&=0.
\end{split}
\end{equation}
\end{theorem}

\begin{remark}[Growth of the function $\sigma=(\sigma_t)_{t\geq 0}$ at infinity]
\hfill

\noindent
Since for any $r>0$ we assume $\lim_{t\to \infty}\frac{\sigma(t+r)}{\sigma(t)}= 1$, it is not hard to see that $\sigma \circ \ln$ is a slowly varying function at infinity.
By the Karamata Representation Theorem
we have that
$\sigma_t=\mathrm{o}(e^{ct})$  as $t\to \infty$ for any $c>0$.
\end{remark}

\noindent The second main result in this subsection states a characterization of profile cut-off  in the total variation distance in terms of the geometric shape of the $\omega$-limit set $\omega(x)$. 
\begin{theorem}[Profile cut-off convergence in the total variation distance]\label{th:ptvgeneral}
\hfill

\noindent
Suppose the assumptions of Theorem~\ref{th:tvgeneral} hold true and recall the notation.
Then the following statements are equivalent. The limit
\begin{equation}\label{e:tv}
\lim\limits_{\ve\to 0^+}
\tv\left(\frac{X^{\ve}_{t^{\mathsf{cut}}_{\ve}+r\cdot w_{\ve}}(x)}{\ve\sigma_{t^{\mathsf{cut}}_{\ve}+r\cdot w_{\ve}}}, Z\right)=\tv\left(\lambda^{1-\ell}e^{-\lambda r w}v(x)+Z,Z\right)\quad \textrm{ holds true}
\end{equation}
for any element $v(x)\in \omega(x)$,
if and only if, for any fixed positive $\rho$ the map 
\begin{equation}\label{e:tvp}
\omega(x)\ni v\mapsto \tv(\rho v+Z,Z)\quad \textrm{ is constant.}   
\end{equation}
\end{theorem}

\subsection{Multivariate cut-off phenomenon in the Wasserstein distance}

The first main result in this subsection states sufficient conditions for the window cut-off in the Wasserstein distance of order $p\geq 1$.
\begin{theorem}[Window cut-off convergence in the Wasserstein distance 
 of order $p\geq 1$]\label{th:Wageneral}
\hfill

\noindent
Assume that there exists a deterministic function $\sigma:(0,\infty)\to (0,\infty)$ and a random vector $Z$ on $\mathbb{R}^m$
satisfying
\begin{equation}
\label{eq:Wanew}
\lim\limits_{t\to \infty} \Wp\left(\frac{S_t}{\sigma_t},Z\right)=0.
\end{equation}
In addition, assume that
for any $r>0$ it follows that
$\lim\limits_{t\to \infty} \frac{\sigma_{t}}{\sigma_{t+r}}=1$.
Let $p\geq 1$ be fixed.
Then for any $(w_{\ve})_{\ve\in(0,1)}$ such that $w_{\ve}\to w$ as $\ve\to 0^+$ with $w>0$ it follows that
\begin{equation}
\begin{split}
\liminf_{\ve\to 0^+}
\Wp\left(\frac{X^{\ve}_{t^{\mathsf{cut}}_{\ve}+r\cdot w_{\ve}}(x)}{\ve\sigma_{t^{\mathsf{cut}}_{\ve}+r\cdot w_{\ve}}}, Z\right)&=\Wp\left(\lambda^{1-\ell}e^{-\lambda r w}\check{v}(x)+Z,Z\right)=\lambda^{1-\ell}e^{-\lambda r w}\|\check{v}(x)\|,\\
\limsup_{\ve\to 0^+}
\Wp\left(\frac{X^{\ve}_{t^{\mathsf{cut}}_{\ve}+r\cdot w_{\ve}}(x)}{\ve\sigma_{t^{\mathsf{cut}}_{\ve}+r\cdot w_{\ve}}}, Z\right)&=\Wp\left(\lambda^{1-\ell}e^{-\lambda r w}\hat{v}(x)+Z,Z\right)=\lambda^{1-\ell}e^{-\lambda r w}\|\hat{v}(x)\|
\end{split}
\end{equation}
for some $\check{v}(x),\hat{v}(x)\in \omega(x)$, where $\omega(x)$ is given in~\eqref{eq:basin}
and $t^{\mathsf{cut}}_{\ve}$
is defined in~\eqref{e:cutti}.
In addition, the following limits holds true:
\begin{equation}
\begin{split}
\lim_{r\to -\infty}\liminf\limits_{\ve\to 0^+}
\Wp\left(\frac{X^{\ve}_{t^{\mathsf{cut}}_{\ve}+r\cdot w_{\ve}}(x)}{\ve\sigma_{t^{\mathsf{cut}}_{\ve}+r\cdot w_{\ve}}}, Z\right)
&=1,\\
\lim_{r\to \infty}\limsup\limits_{\ve\to 0^+}
\Wp\left(\frac{X^{\ve}_{t^{\mathsf{cut}}_{\ve}+r\cdot w_{\ve}}(x)}{\ve\sigma_{t^{\mathsf{cut}}_{\ve}+r\cdot w_{\ve}}}, Z\right)
&=0.
\end{split}
\end{equation}
\end{theorem}

\noindent The second main result in this subsection establishes a characterization of profile cut-off in the Wasserstein distance of order $p\geq 1$.

\begin{theorem}[Profile cut-off convergence in the Wasserstein distance 
 of order $p\geq 1$]\label{th:pWageneral}
\hfill

\noindent
Assume the assumptions of Theorem~\ref{th:Wageneral} holds true and rely on the notation. 
Then the following is equivalent. The limit
\begin{equation}\label{ec:rhs}
\lim\limits_{\ve\to 0^+}
\Wp\left(\frac{X^{\ve}_{t^{\mathsf{cut}}_{\ve}+r\cdot w_{\ve}}(x)}{\ve\sigma_{t^{\mathsf{cut}}_{\ve}+r\cdot w_{\ve}}}, Z\right)=\Wp\left(\lambda^{1-\ell}e^{-\lambda r w}v(x)+Z,Z\right) \quad \textrm{ holds true}
\end{equation}
for any element $v(x)\in \omega(x)$,
if and only if, for any fixed positive $\rho$ the map 
\begin{equation}
\omega(x)\ni v \mapsto \Wp(\rho v+Z,Z)=\rho\|v\| \quad \textrm{ is constant.}   
\end{equation}
\end{theorem}

\begin{remark}[Wasserstein profile for $p\geq 1$ in terms of the exponential function]\hfill

\noindent
By Item~(iii)~in~Lemma~\ref{lem:basicwp} for $p\geq 1$  it follows that the right-hand side of~\eqref{ec:rhs} has an exponential shape. That is, for every $r\in \mathbb{R}$, we have
\[
\Wp\left(\lambda^{1-\ell}e^{-\lambda r w}v(x)+Z,Z\right)=\lambda^{1-\ell}e^{-\lambda r w}\|v(x)\|.
\]
\end{remark}

\section{Examples}\label{sec:examples}

\subsection{Linear equations and convolution processes $S$}\label{subsec: convolution}

Let $D:=(D_t)_{t\geq 0}$ be a continuous-time stochastic process on $\mathbb{R}^m$ (a.k.a. the driver process) such that the stochastic convolution $S:=(S_t)_{t\geq 0}$ given by
\begin{equation}\label{eq:scov}
S_t:=\int_{0}^{t} e^{-\Lambda (t-s)} \ud D_s\quad \textrm{ for all }\quad t\geq 0, 
\end{equation}
is component-wise well-defined (in an appropriate sense). Since for all $t\geq 0$ the components of the map $s\mapsto e^{-\Lambda (t-s)}$ are of bounded variation, it is well-known that the integral 
in~\eqref{eq:scov} is well-defined provided that the noise $D$  has a continuous version. We also note that by using the smoothness of $s\mapsto e^{-\Lambda (t-s)}$, the integral can be defined through integration by parts formula. That is,
\begin{equation*}
S_t= D_t - e^{-\Lambda t}D_0 + \int_{0}^{t} \Lambda e^{-\Lambda (t-s)}D_s\ud s,
\end{equation*}
where the integral is just the classical Lebesgue integral. Finally, note that if $D$ has jumps, one can always consider  the continuous part of $D$ and the jumps separately.
Note that throughout the article we assume that 
$\mathbb{E}[S_t]=0_m$ for all $t\geq 0$, but more general cases can be handled by shifting with $\mathbb{E}[D_t]$ (and having suitable properties for the mean function).

\subsection{Recovering some existing results}

Recently, the cut-off phenomenon has been established for linear and non-linear Langevin equations 
for $\ve$-small Brownian and L\'evy noise in total variation distance and the Wasserstein-$p$ distance in a series of articles~\cite{BJ, BJ1, BPJC, BHPTV, BHPWA, BHPWANO, BHPSPDE, BHPPJSP}. 
We briefly review how the window cut-off results in those papers can be recovered for linear drift (Ornstein--Uhlenbeck type processes) by the verification of  
Condition~\eqref{eq:tvhyp} in 
Theorem~\ref{th:ptvgeneral} and Condition~\eqref{eq:Wanew} in Theorem~\ref{th:Wageneral}.  
In~\cite{BJ} the authors establish the window/profile cut-off phenomenon for scalar non-linear Langevin equations driven by small Brownian motion in the total variation distance via a linearization procedure, Girsanov's Theorem, and explicit Gaussian formulas for the invariant measure. As a by product, profile cut-off is shown for the Ornstein--Uhlenbeck process in total variation due to explicit formulas for the total variation. 
In~\cite{BJ1} these results are generalized to the multivariate Langevin dynamics with coercive non-linearity and additive non-degenerate Brownian motion. 
The proof method boils down to an even more sophisticated linearization procedure, at the core of which, the Ornstein--Uhlenbeck process exhibits a window cut-off phenomenon in total variation distance, which then allows an asymptotic expansion of the limiting measure in $\ve$ by explicit formulas for the total variation distance between multivariate normal distributions. 
In~\cite{BPJC} the authors study the multivariate Ornstein--Uhlenbeck process driven by a $\ve$-small non-degenerate L\'evy process. Here several new features emerge. Note that for $D_t= L_t$ a L\'evy process, we have by path-wise integration by parts and the fact that stationarity and the independence of the increments imply that the reflected process
$\widetilde{L}_r = L_t - L_{t-r}$ has the same finite-dimensional laws as $L_r$, that
\begin{align*}
S_t 
&\stackrel{\mathrm{Law}}{=} \int_0^t e^{-\Lambda u} \ud\widetilde{L}_u.
\end{align*}
This can be carried out by standard calculations, see for 
instance~\cite[Formula~(2.11)]{BH23}. Consequently, 
Condition~\eqref{eq:tvhyp} reads for $\sigma_t \equiv 1$ as 
\begin{align*}
\tv(S_t, Z) = \tv\left(\int_0^t e^{-\Lambda u} \ud L_u, \int_0^\infty e^{-\Lambda u} \ud L_u\right), 
\end{align*}
which is satisfied whenever the laws of $S_t$ and $Z$ are absolutely continuous w.r.t. the Lebesgue measure. However, the smoothness of $Z$ is not always guaranteed for (possibly degenerate) L\'evy drivers, since the total variation distance is in general discontinuous under discrete approximations of absolutely continuous laws, see~\cite[Lemma~1.17]{BHPTV}. Consequently, the authors code sufficient conditions for the underlying smoothness assumption in Condition~(H) there, which is stated in terms of the integrability of the characteristic function $\widehat{S}_t$ of $S_t$ and the uniform convergence of the 
tail integrals $\int_{|\lambda|>R} |\widehat{S}_t(\lambda)| \ud\lambda$ to $0$ as $R\to \infty$. 
In~\cite{BHPTV} these results have been established in a long case study to non-linear Langevin equations driven by  a class of L\'evy process of infinite intensity, called strongly locally layered stable processes, satisfying a so-called equator condition. However, the rather strong condition of the pole at the origin of the L\'evy measure of order $\alpha> \nicefrac{3}{2}$, which appears as a sufficient condition to ensure the respective smoothness, seems difficult to verify. 

On the other hand, for a L\'evy driver $D = L$ Condition~\eqref{eq:Wanew} for the Wasserstein-$p$ distance is easier to establish, even in the case of a degenerate dispersion matrix in front of the driver, since it only depends on the moments of $X^\ve$, which coincide with those of $L$ and do not vary over time. 
A particular role is played by the recently discovered shift-additivity property~\cite[Lemma~2.2]{BHPWA} of the Wasserstein-$p$ distance, since it reduces the complexity of calculations considerably. 
As a consequence, in~\cite[Subsection~4.2.4]{BHPWA} the authors establish window cut-off for the degenerate system of a linear oscillator under subcritical damping and for linear chains of oscillators with heat-baths at outer extremes. These results are generalized to linear infinite dimensional systems, like the heat and the damped wave equation 
in~\cite{BHPSPDE}, and very recently to energy shell lattice models 
in~\cite{BHPPJSP}. In~\cite{BHPWANO} the authors show how to generalize the preceding results to non-linear Langevin equations with strongly coercive drift in the Wasserstein-$p$ distance. 

\subsection{Average process from sampling a fractional Ornstein--Uhlenbeck processes}

For a given $N\in \mathbb{N}$ we consider an i.i.d. sampling array of processes 
$(X^{H,j})_{j=1,\ldots,N}$. More precisely, each
$X^{H,j}=(X^{H,j}_t)_{t\geq 0}$ is an independent copy of
\[
X^{H}_t(x)=e^{-\lambda t}x+ e^{-\lambda t}\int_{0}^{t}e^{\lambda s} \ud B^{H}_s \quad \textrm{ for all }\quad  t\geq 0,
\]
where $(B^{H}_t)_{t\geq 0}$ are fractional Brownian motions with fixed Hurst index $H\in (0,1)$.
We now consider the
average process $S^{H,N}:=(S^{H,N}_t)_{t\geq 0}$ given by $S^{H,N}_t:=\frac{1}{N}\sum_{j=1}^{N}X^{H,j}_t$, $t\geq 0$. We note that 
\[
\ud S^{H,N}_t(x)=
-\lambda S^{H,N}_t \ud t+
\frac{1}{\sqrt{N}} \ud \widetilde{B^{H}_s} \quad \textrm{ for all }\quad  t\geq 0,\qquad S^{H,N}_0=x,
\]
where $(\widetilde{B^{H}_t})_{t\geq 0}$ is a fractional Brownian motion of index $H$. Here the complexity parameter is $\ve=\ve_N:=1/\sqrt{N}$.
Hence, the assumptions of Theorem~\ref{th:tv} and Theorem~\ref{th:Wp} (for all $p\geq 1$) are valid for $(S^{H,N}_t)_{t\geq 0}$ with cut-off time scale $t_{\ve_N}=\frac{1}{2\lambda}\ln(N)$. The standard Brownian driver ($H=\nicefrac{1}{2}$) has been studied in~\cite{BAR2018,Lachaud}.
The same remains true if the fractional Brownian motions $B^H$ are replaced with some other suitable driving noises, cf. subsections below.

\subsection{Generalized Ornstein--Uhlenbeck processes}
Suppose that the noise $D$ has stationary increments with variance function $V(t)$. Then by the arguments of the proof 
of~\cite[Theorem~1]{Voutilainen et alt} the process 
\begin{equation}
\label{eq:stationary-U}
U_t = e^{-\Lambda t}\int_{-\infty}^t e^{-\Lambda(t-s)}\ud D_s\quad \mathrm{ for }\quad t\geq 0
\end{equation}
is a well-defined stationary process. In fact, it follows from elementary change of variables that $U$ is stationary whenever $D$ has stationary increments, provided that for all $t\geq 0$ the integral $\int_{-\infty}^t e^{\Lambda s}\ud D_s$ is well-defined, which follows since $-\Lambda$ is Routh--Hurwitz. 
Note that for $S_t$ defined by 
\begin{equation*}
S_t:=\int_{0}^{t} e^{-\Lambda (t-s)} \ud D_s\quad \textrm{ for all }\quad t\geq 0, 
\end{equation*}
we have the relation
\begin{equation}
\label{eq:SU_relation}
S_t = e^{-\Lambda t}U_0 + U_t\quad \mathrm{ for }\quad t\geq 0.
\end{equation}
Now
\[
\mathrm{Var}(S_t) = \mathbb{E}[S_tS_t^*] = \mathcal{R}_U(0) + e^{-\Lambda t}\mathcal{R}_U(0)e^{-\Lambda^* t}+e^{-\Lambda t}\mathcal{R}_U(t)+\mathcal{R}_U^*(t)e^{-\Lambda^* t},
\]
where $\mathcal{R}_U(t)$ is the covariance function of $U$ given by $\mathcal{R}_U(t) = \mathbb[U_0U_t^*]$. 
For $X^{\ve}_t(x)$ defined by~\eqref{d:X}, we have the link 
\[
X^{\ve}_t(x) = e^{-\Lambda t}(x_0-\ve U_0)+\ve U_t
\quad \mathrm{ for }\quad t\geq 0.
\]
We then immediately obtain the following proposition. 
\begin{proposition}
\label{prop:stationary-convergence}
We have $S_t - U_t \to 0$ almost surely and in $L^p(\Omega)$ for any $p\geq 1$ such that $U \in L^p(\Omega)$. In particular, the statements of Theorem~\ref{th:Wageneral} and Theorem~\ref{th:pWageneral} are true with $\sigma_t\equiv 1$ and $Z\stackrel{\mathrm{Law}}{=} \mu \stackrel{\mathrm{Law}}{=} U_0$.
\end{proposition}   
\begin{proof}
The claimed convergence follows directly 
from~\eqref{eq:SU_relation}, and the statements related to convergence in the Wasserstein distance follow from the fact that convergence in $\Wp$ is equivalent to the convergence of $p$-th absolute moments together with weak convergence. 
\end{proof}

\begin{remark}
Note that the above proposition essentially covers all square integrable stationary processes. Indeed, 
by~\cite[Theorem~1]{Voutilainen et alt} a continuous-time process $U$ is stationary if and only if it is of the 
form~\eqref{eq:stationary-U} for some positive definite $\Lambda$ which implies  that $-\Lambda$ is Routh--Hurwitz. 
\end{remark}

\begin{remark}
Note that by the square integrability of $D$, we can always choose any $p\in[1,2]$, while for higher values $p>2$ one needs the existence of $p$-moments for $p>2$.
\end{remark}

\begin{remark}
While the above result provides weak convergence and convergence in $L^p$ (and hence convergence in the Wasserstein distance of order $p$), these are not sufficient to obtain convergence in the total variation distance, but instead one needs to impose additional assumptions. For example, if the underlying random variables admit densities $f_{S_t}$ and $f_Z$, then $\tv(S_t, Z)\to 0$ as $t\to \infty$ if and only if $f_{S_t}-f_Z \to 0$ in $L^1(\mathbb{R})$. Indeed, from Scheffe's inequality we can deduce
\[
\tv(S_t,Z) = \frac{1}{2}\int_{\mathbb{R}}|f_{S_t}(x)-f_Z(x)|\ud x.
\]
Another possibility to obtain convergence in the total variation distance is to lean on Malliavin calculus, whenever the associated randomness is generated by Gaussian objects, see~\cite{Nualart}. In particular, the convergence in Proposition~\ref{prop:stationary-convergence} takes place in the total variation distance as well whenever $S$ belongs to a finite sum of so-called Wiener chaoses,
see~\cite{Nourdin--Poly} and the references therein. As a particular example, one could consider the Rosenblatt process as the driver $G$, see 
e.g.~\cite{Rosenblatt} for details on the Rosenblatt process. Indeed, as the Rosenblatt process lives in the second Wiener chaos, we also have the convergence in the total variation distance, and the statements of Theorem~\ref{th:tvgeneral}, Theorem~\ref{th:ptvgeneral}, Theorem~\ref{th:Wageneral}, and Theorem~\ref{th:pWageneral}, are true with any $p\geq 1$.
Lastly, we point out that total variation convergence can be obtained via uniform integrability of the characteristic functions, see Proposition~A.3 and Proposition~A.5 in Appendix~A 
of~\cite{BARESQ}.
\end{remark}

\subsection{Iterated Gaussian Ornstein--Uhlenbeck processes}
Let $S$ be given as the stochastic convolution
\[
S_t = \int_0^t e^{-\Lambda (t-s)}\ud D_s \quad \mathrm{ for }\quad t\geq 0,
\]
with centered, continuous, and stationary driver noise $D=(D_s)_{s\geq 0}$. Then, again 
by~\cite[Theorem~1]{Voutilainen et alt}, we can write $D$ as the solution to
\[
\ud D_s = -\Theta D_s \ud s + \ud G_s
\]
for some stationary increment process $G$. As such, $S$ can be viewed as the iterated Ornstein--Uhlenbeck process. In the Gaussian case we have the following result.
\begin{proposition}
Suppose $D$ is a centered, continuous, stationary Gaussian process with covariance matrix $\mathcal{R}_D(s) =\mathbb{E}[D_sD_0^*]$. Then as $t$ tends to infinity it follows that $S_t = \int_0^t e^{-\Lambda(t-s)}\ud D_s \to \cN(0,\Sigma)$ in the total variation distance and in the Wasserstein distance $\Wp$ for any $p\geq 1$, where 
\begin{equation}
\begin{split}
\Sigma &=\mathcal{R}_D(0)+\int_0^\infty \mathcal{R}_D(s)e^{-\Lambda^* s}\Lambda^* \ud s + \int_0^\infty \Lambda e^{-\Lambda s}\mathcal{R}_D(-s)\ud s\\
&\quad
+ \int_0^\infty \int_0^\infty \Lambda e^{-\Lambda s}\mathcal{R}_D(s-v)e^{-\Lambda^* v}\Lambda^* \ud s \ud v.
\end{split}
\end{equation}
In particular, the statements of 
Theorem~\ref{th:tvgeneral},Theorem~\ref{th:ptvgeneral}, Theorem~\ref{th:Wageneral}, and Theorem~\ref{th:pWageneral} are true with $\sigma_t\equiv 1$ and $Z\stackrel{\mathrm{Law}}{=} \cN(0,\Sigma)$.
\end{proposition}
\begin{proof}
Using integration by parts, we obtain
\[
S_t = D_t - e^{-\Lambda t}D_0 + \int_0^t \Lambda ^{-\Lambda (t-s)}D_s \ud s.
\]
Now by Gaussianity, it suffices to show the convergence of the covariance matrix. Moreover, since $e^{-\Lambda t}D_0 \to 0$ as $t\to \infty$, it suffices to consider 
\[
\widetilde{S}_t = D_t + \int_0^t \Lambda ^{-\Lambda (t-s)}D_s \ud s. 
\]
For this we obtain by direct computations that 
\begin{align*}
\mathbb{E}[\widetilde{S}_t\widetilde{S}_t^*] &= \mathbb{E}[D_tD_t^*] + \int_0^t \mathbb{E}[D_tD_s^*]e^{-\Lambda^*(t-s)}\Lambda^* \ud s + \int_0^t \Lambda e^{-\Lambda(t-s)}\mathbb{E}[D_sD_t^*]\ud s \\
&\qquad+ \int_0^t \int_0^t \Lambda e^{-\Lambda s}\mathbb{E}[D_sD_v^*]e^{-\Lambda^* v}\Lambda^* \ud s \ud v \\
&=\mathcal{R}_D(0) + \int_0^t \mathcal{R}_D(t-s)e^{-\Lambda^*(t-s)}\Lambda^* \ud s + \int_0^t \Lambda e^{-\Lambda(t-s)}\mathcal{R}_D(s-t)\ud s \\
&\qquad+ \int_0^t \int_0^t \Lambda e^{-\Lambda s}\mathcal{R}_D(s-v)e^{-\Lambda^* v}\Lambda^* \ud s \ud v \\
&\to \mathcal{R}_D(0)+\int_0^\infty \mathcal{R}_D(s)e^{-\Lambda^* s}\Lambda^* \ud s + \int_0^\infty \Lambda e^{-\Lambda s}\mathcal{R}_D(-s)\ud s \\
&\qquad+ \int_0^\infty \int_0^\infty \Lambda e^{-\Lambda s}\mathcal{R}_D(s-v)e^{-\Lambda^* v}\Lambda^* \ud s \ud v=:\Sigma,
\end{align*}
as $t$ tends to infinity,
where the last convergence can be obtained by using change of variables. This completes the proof.
\end{proof}

\begin{remark}
Note that this result generalizes beyond Gaussian processes $D$. However, one then has to analyze the convergence of both terms in $\widetilde{S}_t$ more carefully. Indeed, while the covariance always converge to $\Sigma$, the weak convergence does not automatically follow since now both terms contribute to the limit. 
\end{remark}

\subsection{Non-homogeneous example}
In this subsection, we study linear non-homogeneous process of the type
\[
X^{\ve}_t=e^{-\lambda t}x+\ve e^{-\lambda t}\int_{0}^{t}e^{\lambda s}  \tau(s) \ud B_s,\quad t\geq 0,\, x\neq 0,
\]
where $\tau:[0,\infty)\to (0,\infty)$ is a continuous deterministic function and $(B_t)_{t\geq 0}$ is a standard Brownian motion.
By It\^o's isometry we have 
$X^\ve_t\stackrel{\mathrm{Law}}{=}\cN(e^{-\lambda t}x,e^{-2\lambda t}\int_{0}^{t}e^{2\lambda s}  (\tau(s))^2 \ud s)$. Provided that $\lim_{t\to\infty} \tau(t) = \tau(\infty) \in (0,\infty)$, L'H\^opital's rule implies
\[
e^{-2\lambda t}\int_{0}^{t}e^{2\lambda s}  (\tau(s))^2 \ud s \to \frac{\tau^2(\infty)}{2\lambda}\quad \mathrm{ as } \quad t\to \infty.
\]
From this, it follows that the statements of Theorem~\ref{th:tvgeneral}, Theorem~\ref{th:ptvgeneral}, Theorem~\ref{th:Wageneral}, and Theorem~\ref{th:pWageneral} are true with $\sigma_t\equiv 1$ and $Z\stackrel{\mathrm{Law}}{=} \cN\left(0,\frac{\tau^2(\infty)}{2\lambda}\right)$.

\subsection{Ex-centric integrated Ornstein--Uhlenbeck process}

In this subsection, we study linear equations driven by  Ornstein--Uhlenbeck processes.
It serves as a model for the position of a linear oscillator with linear friction perturbed by a Brownian or a symmetric $\alpha$-stable noise, see~\cite{HP2014}. This example yields a meaningful example, where~\eqref{eq:tvhyp} 
and~\eqref{eq:Wanew} are satisfied with a non-constant scale $\sigma_t \not \equiv 1$.

\noindent Let $Y^{\ve}_t=y+\int_{0}^{t} X^{\ve}_s \ud s$, $t\geq 0$, $y\in \mathbb{R}$ be the integrated Ornstein--Uhlenbeck process driven by 
\[
X^{\ve}_s=e^{-\lambda s}x+\ve e^{-\lambda s}\int_{0}^{s}e^{\lambda r} \ud B_r,\quad s\geq 0,\, x\neq 0,
\]
where $\lambda>0$ and $\ve>0$. 
Using the stochastic Fubini Theorem  we obtain that 
\begin{equation}
\begin{split}
\int_{0}^{t} \left(e^{-\lambda s}\int_{0}^{s}e^{\lambda r} \ud B_r\right) \ud s&=
\int_{0}^{t} \frac{1}{\lambda}\left(1-e^{-\lambda (t-r)}\right)  \ud B_r.
\end{split}
\end{equation}
By the It\^o isometry we obtain 
\[
Y^{\ve}_t\stackrel{\mathrm{Law}}{=}\cN\left(y+\frac{1}{\lambda}(1-e^{-\lambda t})x,
\frac{\ve^2}{\lambda^2}
\left(
t-\frac{2}{\lambda}(1-e^{-\lambda t})+\frac{1}{2\lambda}(1-e^{-2\lambda t})\right)
\right).
\]
Define $Z^{\ve}_t:=\frac{Y^{\ve}_t-(y+x/\lambda)}{\sqrt{t}}$ and
note that 
$Z^{\ve}_t$ has law
\[
Z^{\ve}_t\stackrel{\mathrm{Law}}{=}\cN\left(-\frac{1}{\lambda}\frac{e^{-\lambda t}}{\sqrt{t}}x,
\frac{\ve^2}{\lambda^2}
\left(
1-\frac{2}{\lambda t}(1-e^{-\lambda t})+\frac{1}{2\lambda t}(1-e^{-2\lambda t})\right)
\right)
\]
and its limiting law as $t\to \infty$ is given by
$
Z^{\ve}_\infty\stackrel{\mathrm{Law}}{=}\cN\left(0,
\frac{\ve^2}{\lambda^2}
\right)$. Hence the hypotheses of Theorem~\ref{th:ptvgeneral} and Theorem~\ref{th:pWageneral} (for all $p\geq 1$) are valid for $S_t=Y^{\ve}_t-(y+x/\lambda)$ and nontrivial $\sigma_t=\sqrt{t}$.
If we replace the driver $B$ of $X^{\ve}$ by a symmetric $\alpha$-stable driver $L$, $\alpha\in (1,2)$ we  see that
\[
Y^{\ve}_t\stackrel{\mathrm{Law}}{=}
y+\frac{1}{\lambda}(1-e^{-\lambda t})x+
\frac{\ve}{\lambda}
\int_{0}^{t} (1-e^{-\lambda (t-r)})  \ud L_r,
\]
where the preceding stochastic integral has an $\alpha$-stable law whose characteristic function is given by
\begin{equation}
\psi(z)=
\exp\left(-|z|^{\alpha} C_{\alpha}\frac{\ve^{\alpha}}{\lambda^\alpha}\int_{0}^t (1-e^{-\lambda r})^\alpha \ud r\right),\quad z\in \mathbb{R},
\end{equation}
where $C_\alpha$ is a positive constant, see~(1.11)~and~(1.13) in~\cite{Kyprianou--Pardo}.
By L'H\^opital's rule we have 
$
\lim\limits_{t\to \infty}\frac{1}{t}\int_{0}^t (1-e^{-\lambda r})^\alpha \ud r=1$.
For $Z^{\ve}_t:=\frac{Y^{\ve}_t-(y+x/\lambda)}{t^{1/\alpha}}$ we obtain
\begin{equation}
Z^{\ve}_t\stackrel{\mathrm{Law}}{=}
-\frac{1}{\lambda}\frac{e^{-\lambda t}}{t^{1/\alpha}}x+
\frac{\ve}{\lambda t^{1/\alpha}}
\int_{0}^{t} (1-e^{-\lambda (t-r)})  \ud L_r
\end{equation}
and its limiting law is given by $
Z^{\ve}_\infty\stackrel{\mathrm{Law}}{=} \ve L_1$,
where $L_1$ has characteristic function
\begin{equation}
\phi(z)=
\exp\left(-|z|^{\alpha} \frac{C_{\alpha}}{\lambda^\alpha}\right),\quad z\in \mathbb{R}.
\end{equation}
Whence the hypotheses of Theorem~\ref{th:ptvgeneral} and Theorem~\ref{th:pWageneral} (for $1\leq p<\alpha$) are valid for $S_t=Y^{\ve}_t-(y+x/\lambda)$ and $\sigma_t=t^{1/\alpha}$. 
\begin{remark}
By more tedious computations and using integration by parts, the above results generalize to the case of more general noises $D$ instead of only Brownian noises $W$ or $\alpha$-stable noises $L$.
\end{remark}

\begin{appendix}

\section{Proofs of the main results}\label{ap:proofs}

In this section, we provide the proofs of Theorem~\ref{th:tvgeneral} and Theorem~\ref{th:ptvgeneral}. The proofs of Theorem~\ref{th:Wageneral} and Theorem~\ref{th:pWageneral} are analogous and hence we left the details to the reader.
\begin{proof}[Proof of Theorem~\ref{th:tvgeneral} (Window cut-off convergence in the total variation distance)]
\hfill

\noindent
Let $x\in \mathbb{R}^m\setminus\{0_m\}$ be fixed. For each $\ve>0$ and $t>0$, we define 
\begin{equation}\label{eq:cero}
\begin{split}
\ud_{\mathrm{TV}}(t;\ve,x):=\tv\left(\frac{X^{\ve}_{t}(x)}{\ve \sigma_t},  Z\right)=
\tv\left(
\frac{e^{-\Lambda t}x}{\ve \sigma_t}+\frac{S_t}{\sigma_t}, Z\right).
\end{split}
\end{equation}
By the triangle inequality 
and 
Item~(i) in Lemma~\ref{lem:ptv}  we get
\begin{equation}\label{equ0}
\begin{split}
\tv\left(
\frac{e^{-\Lambda t}x}{\ve \sigma_t}+\frac{S_t}{\sigma_t}, Z\right)&\leq 
\tv\left(
\frac{e^{-\Lambda t}x}{\ve \sigma_t}+\frac{S_t}{\sigma_t}, 
\frac{e^{-\Lambda t}x}{\ve \sigma_t}+Z\right)+
\tv\left( 
\frac{e^{-\Lambda t}x}{\ve \sigma_t}+Z,Z\right)\\
&\leq \tv\left(
\frac{S_t}{\sigma_t}, 
Z\right)+
\tv\left( 
\frac{e^{-\Lambda t}x}{\ve \sigma_t}+Z,Z\right).
\end{split}
\end{equation}
Similarly, we obtain
\begin{equation}\label{equ1}
\begin{split}
\tv\left(
\frac{e^{-\Lambda t}x}{\ve \sigma_t}+Z, Z\right)&\leq 
\tv\left(
\frac{e^{-\Lambda t}x}{\ve \sigma_t}+Z, 
\frac{e^{-\Lambda t}x}{\ve \sigma_t}+\frac{S_t}{\sigma_t}\right)+
\tv\left( 
\frac{e^{-\Lambda t}x}{\ve \sigma_t}+\frac{S_t}{\sigma_t},Z\right)\\
&\leq \tv\left(
\frac{S_t}{\sigma_t}, 
Z\right)+
\tv\left( 
\frac{e^{-\Lambda t}x}{\ve \sigma_t}+\frac{S_t}{\sigma_t},Z\right).
\end{split}
\end{equation}
By~\eqref{eq:cero},~\eqref{equ0}~and~\eqref{equ1} we get
\begin{equation}
\left|\ud_{\mathrm{TV}}(t;\ve,x)-\tv\left( 
\frac{e^{-\Lambda t}x}{\ve \sigma_t}+Z,Z\right)\right|\leq \tv\left(
\frac{S_t}{\sigma_t}, 
Z\right),
\end{equation}
and hence, for any function $(\tf_{\ve})_{\ve>0}$  such that $\tf_{\ve}\to \infty$ as $\ve \to 0$, we have
\begin{equation}
\begin{split}
\liminf_{\ve\to 0^+}
\ud_{\mathrm{TV}}(\tf_\ve;\ve,x)&=\liminf_{\ve\to 0^+}\tv\left( 
\frac{e^{-\Lambda {\tf_\ve}}x}{\ve \sigma_{\tf_\ve}}+Z,Z\right),\\
\limsup_{\ve\to 0^+}
\ud_{\mathrm{TV}}(\tf_\ve;\ve,x)&=\limsup_{\ve\to 0^+}\tv\left( 
\frac{e^{-\Lambda {\tf_\ve}}x}{\ve \sigma_{\tf_\ve}}+Z,Z\right)
\end{split}
\end{equation}
due to~\eqref{eq:tvhyp}.
By Lemma~\ref{lem:Har} we have
\begin{equation}\label{eq:limitHar}
\lim\limits_{t\to \infty}
\left\|
\frac{e^{\lambda t}}{t^{\ell-1}}e^{-\Lambda t}x-v(t;x)  \right\|=0,
\end{equation} 
where $(v(t;x))_{t\geq 0}$
is given in~\eqref{ec:vdef}.
For any $\ve>0$ and $t>0$, we define
\begin{equation}
v_{\ve}(t;x):=\frac{t^{\ell-1}e^{-\lambda t}}{\ve \sigma_t}v(t;x).
\end{equation}
By the triangle inequality
and 
Item~(i)~in~Lemma~\ref{lem:ptv} 
we have
\begin{equation}\label{des1}
\begin{split}
\tv\left( 
\frac{e^{-\Lambda t}x}{\ve \sigma_{t}}+Z,Z\right)&\leq \tv\left( 
\frac{e^{-\Lambda t }x}{\ve \sigma_{t}}+Z,v_{\ve}(t;x)+Z\right)
+\tv\left(v_{\ve}(t;x)+Z,Z\right)\\
&=\tv\left( 
\frac{e^{-\Lambda t }x}{\ve \sigma_{t}}-v_{\ve}(t;x)+Z,Z\right)
+\tv\left(v_{\ve}(t;x)+Z,Z\right)
\end{split}
\end{equation}
for any $\ve>0$ and $t>0$.
Analogously, for any $\ve>0$ and $t>0$, we get
\begin{equation}\label{des2}
\begin{split}\tv\left(v_{\ve}(t;x)+Z,Z\right)
&\leq \tv\left( 
v_{\ve}(t;x)+Z,\frac{e^{-\Lambda t }x}{\ve \sigma_{t}}+Z\right)
+\tv\left(\frac{e^{-\Lambda t}x}{\ve \sigma_{t}}+Z,Z\right)\\
&= \tv\left( 
\frac{e^{-\Lambda t} x}{\ve \sigma_{t}}-v_{\ve}(t;x)+Z,Z\right)
+\tv\left(\frac{e^{-\Lambda t }x}{\ve \sigma_{t}}+Z,Z\right).
\end{split}
\end{equation}
By~\eqref{des1}~and~\eqref{des2}, for any $\ve>0$ and $t>0$, we obtain
\begin{equation}\label{ec:eque}
\begin{split}
\left|
\tv\left(v_{\ve}(t;x)+Z,Z\right)-
\tv\left(\frac{e^{-\Lambda t }x}{\ve \sigma_{t}}+Z,Z\right)\right|
&\leq  \tv\left( 
\frac{e^{-\Lambda t} x}{\ve \sigma_{t}}-v_{\ve}(t;x)+Z,Z\right).
\end{split}
\end{equation}
For any $t>0$
observe that
\begin{equation}
\frac{e^{-\Lambda t} x}{\ve \sigma_{t}}-v_{\ve}(t;x)=\frac{t^{\ell-1}e^{-\lambda t}}{\ve \sigma_t}
\left(
\frac{e^{\lambda t}}{t^{\ell-1}}e^{-\Lambda t}x-v(t;x)  \right)
\end{equation}
and 
\begin{equation}\label{ec:limt}
\lim\limits_{\ve \to 0}\frac{(t^{\mathsf{cut}}_{\ve}+r\cdot w_{\ve})^{\ell-1}e^{-\lambda (t^{\mathsf{cut}}_{\ve}+r\cdot w_{\ve})}}{\ve \sigma_{t^{\mathsf{cut}}_{\ve}+r\cdot w_{\ve}}}=\lambda^{1-\ell}e^{-\lambda r w}\quad \textrm{ for any }\quad r\in \mathbb{R}.
\end{equation}
Inequality~\eqref{ec:eque},
 with the help of~\eqref{eq:limitHar}
and Item~(iii)~in~Lemma~\ref{lem:ptv}, gives
\begin{equation}
\begin{split}
\liminf_{\ve\to 0^+}
\ud_{\mathrm{TV}}(t^{\mathsf{cut}}_{\ve}+r\cdot w_{\ve};\ve,x)&=
\liminf_{\ve\to 0^+}
\tv\left( 
\frac{e^{-\Lambda (t^{\mathsf{cut}}_{\ve}+r\cdot w_{\ve})}x}{\ve \sigma_{t^{\mathsf{cut}}_{\ve}+r\cdot w_{\ve}}}+Z,Z\right)\\
&=\liminf_{\ve\to 0^+}\tv\left(v_{\ve}(t^{\mathsf{cut}}_{\ve}+r\cdot w_{\ve};x)+Z,Z\right)
\end{split}
\end{equation}
and
\begin{equation}
\begin{split}
\limsup_{\ve\to 0^+}
\ud_{\mathrm{TV}}(t^{\mathsf{cut}}_{\ve}+r\cdot w_{\ve};\ve,x)&=
\limsup_{\ve\to 0^+}
\tv\left( 
\frac{e^{-\Lambda (t^{\mathsf{cut}}_{\ve}+r\cdot w_{\ve})}x}{\ve \sigma_{t^{\mathsf{cut}}_{\ve}+r\cdot w_{\ve}}}+Z,Z\right)\\
&=\limsup_{\ve\to 0^+}\tv\left(v_{\ve}(t^{\mathsf{cut}}_{\ve}+r\cdot w_{\ve};x)+Z,Z\right).
\end{split}
\end{equation}
By~\eqref{eq:novacio} we have that $\omega(x)\neq \emptyset$, where  $\omega(x)$ is given in~\eqref{eq:basin}.
For each $r\in \mathbb{R}$ fixed,
by definition of $\liminf\limits_{\ve \to 0}$ we have the existence of a sequence $(\ve_n)_{n\in \mathbb{N}}$ such that 
$\ve_n\to 0$ as $n\to \infty$ and
\begin{equation}
\begin{split}
\liminf_{\ve\to 0^+}\tv\left(v_{\ve}(t^{\mathsf{cut}}_{\ve}+r\cdot w_{\ve};x)+Z,Z\right)
=
\lim\limits_{n\to \infty}\tv\left(v_{\ve_n}(t^{\mathsf{cut}}_{\ve_n}+r\cdot w_{\ve_n};x)+Z,Z\right).
\end{split}
\end{equation}
Note that $(v_{\ve_n}(t^{\mathsf{cut}}_{\ve_n}+r\cdot w_{\ve_n};x))_{n\in \mathbb{N}}$ is a bounded sequence. Hence, the Bolzano--Weierstrass Theorem and~\eqref{ec:limt}
yield 
the existence of a subsequence $(\ve_{n_k})_{k\in \mathbb{N}}$ of 
$(\ve_n)_{n\in \mathbb{N}}$ such that 
$\ve_{n_k}\to 0$ as $k\to \infty$ and 
\begin{equation}
\lim\limits_{k\to \infty }v_{\ve_{n_k}}(t^{\mathsf{cut}}_{\ve_{n_k}}+r\cdot w_{\ve_{n_k}};x)=
\lambda^{1-\ell}e^{-\lambda r w}\check{v}(x)
\end{equation}
for some $\check{v}(x)\in \omega(x)$.
Therefore, Item~(iii)~in~Lemma~\ref{lem:ptv}  implies
\begin{equation}\label{e:inferior}
\begin{split}
\liminf_{\ve\to 0^+}
\ud_{\mathrm{TV}}(t^{\mathsf{cut}}_{\ve}+r\cdot w_{\ve};\ve,x)&=
\lim\limits_{n\to \infty}\tv\left(v_{\ve_n}(t^{\mathsf{cut}}_{\ve_n}+r\cdot w_{\ve_n};x)+Z,Z\right)\\
&=\lim\limits_{k\to \infty}\tv\left(v_{\ve_{n_k}}(t^{\mathsf{cut}}_{\ve_{n_k}}+r\cdot w_{\ve_{n_k}};x)+Z,Z\right)\\
&=
\tv\left(\lambda^{1-\ell}e^{-\lambda r w}\check{v}(x)+Z,Z\right).
\end{split}
\end{equation}
Similarly,
\begin{equation}
\begin{split}
\limsup_{\ve\to 0^+}
\ud_{\mathrm{TV}}(t^{\mathsf{cut}}_{\ve}+r\cdot w_{\ve};\ve,x)=\tv\left(\lambda^{1-\ell}e^{-\lambda r w}\hat{v}(x)+Z,Z\right)
\end{split}
\end{equation}
for some $\hat{v}(x)\in \omega(x)$. Therefore, we conclude~\eqref{e:infsup}.
Recall that 
$0_m\not \in\omega(x)$, see~\eqref{eq:nozero}.
Therefore, Item~(iii) and Item~(iv)~in~Lemma~\ref{lem:ptv} give~\eqref{eq:perfiltv1}.
\end{proof}

\begin{proof}[Proof of Theorem~\ref{th:ptvgeneral} (Profile cut-off convergence in the total variation distance)]\hfill

\noindent
We first assume that~\eqref{e:tv} holds true.
Repeating step by step the proof of~\eqref{e:inferior}  one can deduce that the set of accumulation points of 
\[
\left(\tv\left(\frac{X^{\ve}_{t^{\mathsf{cut}}_{\ve}+r\cdot w_{\ve}}(x)}{\ve\sigma_{t^{\mathsf{cut}}_{\ve}+r\cdot w_{\ve}}}, Z\right)\right)_{\ve>0}
\] 
is given by
$
\{\tv(\lambda^{1-\ell}e^{-\lambda r w}\widetilde{v}(x)+Z,Z):\widetilde{v}(x)\in \omega(x)\}$, which, with the help of~\eqref{e:inferior}, yields~\eqref{e:tvp}.

\noindent We then assume that~\eqref{e:tvp} holds true. Then~\eqref{e:infsup} implies~\eqref{e:tv}. 
\end{proof}

\section{Basic general properties of the total variation distance}\label{ap:tv}

In this section, we recall the definition of the total variation distance.
Let $(\Omega,\mathcal{F})$ be a measurable space and let
$
\mathfrak{P}:=\{\mathbb{P}:\mathcal{F}\to [0,1]: \mathbb{P} \textrm{ is a probability measure}\}.
$
The total variation distance $\tv:\mathfrak{P}\times \mathfrak{P}\to [0,1]$ is defined by
$
\tv(\mathbb{P}_1,\mathbb{P}_2):=\sup_{F\in \mathcal{F}}|\mathbb{P}_1(F)-\mathbb{P}_2(F)| \textrm{ for any }  \mathbb{P}_1,\mathbb{P}_2\in \mathfrak{P}.
$
Let $(\Omega,\mathcal{F},\mathbb{P})$ be a probability space, $m\in \mathbb{N}$, and
let $X_1:\Omega\to \mathbb{R}^m$ and $X_2:\Omega\to \mathbb{R}^m$ be two random vectors having distribution $\mathbb{P}_{X_1}$ and $\mathbb{P}_{X_2}$, respectively.
Then the total variation distance between the law of $X_1$ and the law of $X_2$ is given by  
$
\tv(\mathbb{P}_{X_1},\mathbb{P}_{X_2})=\sup_{B\in \mathcal{B}(\mathbb{R}^m)}|\mathbb{P}(X_1\in B)-\mathbb{P}(X_2\in B)|,
$
where $\mathcal{B}(\mathbb{R}^m)$ denotes the Borelian $\sigma$-algebra of $\mathbb{R}^m$.
In conscious abuse of notation and for short-hand notation, we write $\tv(X_1,X_2)$ instead of $\tv(\mathbb{P}_{X_1},\mathbb{P}_{X_2})$.
The following lemma states the basic properties of the total variation distance.
In the absolute continuous case, the proofs of Item~(i) and Item~(ii) can be found in 
Lemma~A.1~in~Appendix~A of~\cite{BPJC} or Theorem~5.2 in~\cite{Devroye--Lugosi} for the general case.
Item~(iii) follows from the continuity of the shift operator in $L^1(\mathbb{R}^m)$ and Item~(iv) follows from a standard approximation argument.

\begin{lemma}[Basic properties of the total variation distance]\label{lem:ptv}
\hfill

\noindent
Let $(\Omega,\mathcal{F},\mathbb{P})$ be a probability space, and let $X_1:\Omega\to \mathbb{R}^m$ and $X_2:\Omega\to \mathbb{R}^m$ be two random vectors. 
Denote the Euclidean norm in $\mathbb{R}^m$ by
 $\|\cdot\|$. 
Then the following holds true.
\begin{itemize}
\item[(i)] Translation invariant: for any deterministic vectors $v_1,v_2\in \mathbb{R}^m$ it follows that 
\[
\tv(v_1+X_1,v_2+X_2)=\tv((v_1-v_2)+X_1,X_2)=\tv(X_1,(v_2-v_1)+X_2).
\]
\item[(ii)] 
Zero-homogeneity:
for any non-zero constant $c$ it follows that
\[
\tv(cX_1,cX_2)=\tv(X_1,X_2).
\]
\item[(iii)] 
Continuity at the shift: for any function of deterministic vectors $(v_{\ve})_{\ve}$ such that $\|v_{\ve}-v\|\to 0$ as $\ve\to 0^+$
for some $v\in \mathbb{R}^m$,
it follows that
\[
\lim\limits_{\ve \to 0^+}\tv(v_{\ve}+X_1,X_1)=\tv(v+X_1,X_1).
\]
\item[(iv)]
Displacement at infinity: for any function of deterministic vectors $(v_{\ve})_{\ve}$ such that $\|v_{\ve}\|\to \infty$ as $\ve\to 0^+$
it follows that
\[
\lim\limits_{\ve \to 0^+}\tv(v_{\ve}+X_1,X_1)=1.
\]
\end{itemize}
\end{lemma}

\section{Basic general properties of the Wasserstein distance }\label{ap:wp}

In this section, we recall the definition of  Wasserstein distances.
We continue with the notation introduced in Appendix~\ref{ap:tv}.
Let $\mathcal{B}(\mathbb{R}^m)\otimes \mathcal{B}(\mathbb{R}^m)$ be the product $\sigma$-algebra.
We say that a probability measure $\nu:\mathcal{B}(\mathbb{R}^m)\otimes \mathcal{B}(\mathbb{R}^m)\to [0,1]$ is a coupling between the probabilities $\mathbb{P}_{X_1}$ and $\mathbb{P}_{X_2}$ if and only if for all $B\in \mathcal{B}(\mathbb{R}^m)$  we have
$
\nu(B\times \mathbb{R}^m)=\mathbb{P}_{X_1}(B) \textrm{ and } 
\nu(\mathbb{R}^m\times B)=\mathbb{P}_{X_2}(B).
$
We then denote the set of all coupling between $\mathbb{P}_{X_1}$ and $\mathbb{P}_{X_2}$ by $\mathcal{C}(\mathbb{P}_{X_1},\mathbb{P}_{X_2})$.
In addition, assume that
$X_1$ and $X_2$ have
finite absolute $p$-th moments for some $p\geq 1$. More precisely, 
$
\int_{\mathbb{R}^m} \|u\|^p\, \mathbb{P}_{X_1}(\ud u)<\infty$ and
$\int_{\mathbb{R}^m} \|u\|^p\, \mathbb{P}_{X_2}(\ud u)<\infty$,
where $\|\cdot\|$ is the Euclidean norm in $\mathbb{R}^m$.
We define the 
Wasserstein distance of order $p$ between $\mathbb{P}_{X_1}$ and $\mathbb{P}_{X_2}$ by
\begin{equation}
\Wp(\mathbb{P}_{X_1},\mathbb{P}_{X_2}):=\inf_{\nu \in \mathcal{C}(\mathbb{P}_{X_1},\mathbb{P}_{X_2})} 
\left(\int_{\mathbb{R}^m\times \mathbb{R}^m} \|u_1-u_2\|^p \nu(\ud u_1, \ud u_2)\right)^{1/p}.
\end{equation}
In a conscious abuse of notation and for short-hand notation, we write $\Wp(X_1,X_2)$ instead of $\Wp(\mathbb{P}_{X_1},\mathbb{P}_{X_2})$.
The following lemma states the basic properties of $\Wp$, except for non-standard shift-additivity property given in Item~(iii).
The proofs of Item~(i), Item~(ii) and Item~(iii) can be found for instance in Lemma~2.2 in~\cite{BHPWA}, while the proof of Item~(iv) follows using the synchronous coupling and the sub-multiplicative property of the norm.
For further details and
properties of the Wasserstein distances, we refer to~\cite{Panaretos,Villani}.

\begin{lemma}[Basic properties of the Wasserstein distance of order $p\geq 1$]
\label{lem:basicwp}
\hfill

\noindent
Let $(\Omega,\mathcal{F},\mathbb{P})$ be a probability space, and let $X_1:\Omega\to \mathbb{R}^m$ and $X_2:\Omega\to \mathbb{R}^m$ be two random vectors
having finite absolute $p$-th moments for some $p\geq 1$.
Then the following holds true.
\begin{itemize}
\item[(i)] 
Translation invariant:
for any deterministic vectors $v_1,v_2\in \mathbb{R}^m$ it follows that 
\[
\Wp(v_1+X_1,v_2+X_2)=\Wp((v_1-v_2)+X_1,X_2)=\Wp(X_1,(v_2-v_1)+X_2).
\]
\item[(ii)] 
One-homogeneity:
for any non-zero constant $c$ it follows that
\[
\Wp(cX_1,cX_2)=|c|\,\Wp(X_1,X_2).
\]
\item[(iii)] 
Shift-additivity:
for any deterministic vector $v_1\in \mathbb{R}^d$ it follows that
\[
\Wp(v_1+X_1,X_1)=\|v_1\|.
\]
\item[(iv)] 
Shift and scaling continuity:
for any deterministic vectors $v_1,v_2\in \mathbb{R}^d$ and matrices $M_1,M_2\in \mathbb{R}^{m\times m}$ it follows that
\[
\Wp(v_1+M_1X_1,v_2+M_2X_1)\leq \|v_1-v_2\|+\|M_1-M_2\|_2(\mathbb{E}[\|X_1\|^p])^{1/p},
\]
where for $C=(c_{j,k})_{j,k\in \{1,\ldots,m\}}\in \mathbb{R}^{m\times m}$ we set $\|C\|_2:=\left(\sum_{j=1}^{m}\sum_{k=1}^{m}c^2_{j,k}\right)^{1/2}$.
\end{itemize}
\end{lemma}
\end{appendix}

\section*{Declarations}
\noindent
\textbf{Acknowledgments}:
Gerardo Barrera would like to express his gratitude to Aalto University School of Science (Espoo, Finland) and Instituto Superior T\'ecnico (Lisbon, Portugal) for all the facilities used along the realization of this work.

\noindent
\textbf{Ethical approval}: Not applicable.

\noindent
\textbf{Competing interests}: The authors declare that they have no conflict of interest.

\noindent
\textbf{Authors' contributions}:
All authors have contributed equally to the paper.

\noindent
\textbf{Availability of data and materials}: Data sharing not applicable to this article as no data-sets were generated or analyzed during the current study.

\noindent
\textbf{Funding}:
Pauliina Ilmonen and Gerardo Barrera gratefully acknowledge support from the Academy of Finland via Finnish Centre of Excellence in Randomness
and Structures 
(decision number~346308). 
The research of Michael A. H\"ogele has been supported by the project "Mean deviation frequencies and the cutoff phenomenon" (INV-2023-162-2850) of the School of Sciences (Facultad de Ciencias) at Universidad de los Andes, Bogot\'a, Colombia.
The research of Gerardo Barrera, Michael A. H\"ogele and Lauri Viitasaari is partially supported by European Union’s Horizon Europe research and innovation programme under the Marie Sk\l{}odowska-Curie Actions Staff Exchanges (Grant agreement No.~101183168 -- LiBERA, Call: HORIZON-MSCA-2023-SE-01).

\noindent 
\textbf{Disclaimer}: Funded by the European Union. Views and opinions expressed are however those of the author(s) only and do not necessarily reflect those of the European Union or the European Education and Culture Executive Agency (EACEA). Neither the European Union nor EACEA can be held responsible for them.

\end{document}